\documentclass[a4paper,11pt]{amsart}

\usepackage{mathrsfs,nccmath,amssymb,amsthm,mathabx}
\usepackage[dvips]{graphicx,psfrag}
\usepackage[all]{xy}
\xyoption{line}
\usepackage{color}
\usepackage{ulem}

\renewcommand{\le}{\varleq}
\renewcommand{\ge}{\vargeq}

\newcommand\mL{L\kern-0.08cm\char39}

\newcommand{\myand}{\text{ and }}

\newcommand{\seb}{\{\,}
\newcommand{\sen}{\,\}}

\newcommand{\getsby}[1]{\xleftarrow{#1}}

\newcommand{\funcdecomp}[2]{{#1\/}^{\/#2}}

\newcommand{\tesgh}{edge-surjective graph homomorphism}

\newcommand{\pdirectional}{\raise0.05em\hbox{$+$}directional}
\newcommand{\bidirectional}{bidirectional}

\newcommand{\Z}{\mathbb{Z}}
\newcommand{\Nonne}{\mathbb{N}}
\newcommand{\Posint}{\Nonne^+}

\newcommand{\beposint}{\in \Posint}
\newcommand{\bpi}{\beposint} 

\newcommand{\benonne}{\in \Nonne}
\newcommand{\bni}{\in \Nonne} 

\newcommand{\diam}{{\rm diam}}

\newcommand{\pstrzinf}[1]{(#1_0,#1_1,#1_2,\dotsc)}

\newcommand{\Decomp}{\mathscr{D}}

\newcommand{\Gcal}{\mathcal{G}}

\newcommand{\Ucal}{\mathcal{U}}

\newcommand{\kuu}{\emptyset}
\newcommand{\nekuu}{\neq \kuu}
\newcommand{\fai}{\varphi}

\newcommand{\enumb}{\begin{enumerate}}
\newcommand{\enumn}{\end{enumerate}}

\newcommand{\itemb}{\begin{itemize}}
\newcommand{\itemn}{\end{itemize}}

\newtheorem{thm}{Theorem}[section]
\newtheorem{lem}[thm]{Lemma}

\theoremstyle{definition}
\newtheorem{defn}[thm]{Definition}

\theoremstyle{remark}
\newtheorem{nota}[thm]{Notation}

\numberwithin{equation}{section}

\newcommand{\remn}{{\rm remn}}
\newcommand{\gap}{{\rm gap}}
\newcommand{\pinfty}{+\infty}

\begin{document}

\title[Construction of a completely scrambled system by graph covers]
{The construction of a completely scrambled system\\ by graph covers}

\author{TAKASHI SHIMOMURA}

\address{Nagoya University of Economics, Uchikubo 61-1, Inuyama 484-8504, Japan}
\curraddr{}
\email{tkshimo@nagoya-ku.ac.jp}
\thanks{}



\date{\today}

\dedicatory{}

\commby{}

\begin{abstract}
In this paper, we define a new construction of completely scrambled 0-dimensional systems using the inverse limit of sequences of directed graph covers.
These examples are transitive and are not locally equicontinuous.
\end{abstract}

\maketitle
\section{Introduction}
Let $X$ be a compact metrizable space,
 and $f : X \to X$ be a continuous surjective map.
In this paper, we call $(X,f)$ a topological dynamical system,
 and consider the case in which $X$ is 0-dimensional.
In this case, we call $(X,f)$ a 0-dimensional system.
If $X$ is homeomorphic to the Cantor set, $(X,f)$ is called
 a Cantor system.
Akin, Glasner and Weiss \cite{AGW}
 made use of a special sequence of directed graph covers
 to construct a special homeomorphism
 that has the generic conjugacy class in the space of all Cantor systems,
 while Gambaudo and Martens \cite{GM}
 employed special sequences of directed graph covers
 to study ergodic measures of Cantor minimal systems.
In \cite{Shimomura4},
 we generalized this latter construction to arbitrary 0-dimensional systems.
In this paper, we use a sequence of graph covers to construct
 examples that are transitive, completely scrambled, and not locally equicontinuous.
A subset $S \subseteq X$ is called a {\it scrambled set} if, for every $x \ne y \in S$,
\[ \limsup_{n \to \pinfty}d(f^n(x),f^n(y)) > 0\]
and
\[ \liminf_{n \to \pinfty}d(f^n(x),f^n(y)) = 0. \]
Since Li and Yorke  developed the notion of scrambled sets
 in the study of chaotic systems \cite{LY},
 there has been some discussion as to how large such sets can be.
In 1997, Mai reported a non-compact example that is completely scrambled \cite{Mai},
 i.e., the scrambled set is the whole space,
 and conjectured that there was no compact example.
Huang and Ye \cite{HY} later disputed this conjecture.
They constructed a compact, 0-dimensional completely scrambled system.
By taking the product of
 the identity map with any other compact set, and collapsing some subspaces to a point,
 their example indicated the existence of others on a variety of spaces.
In the same paper, they also announced the existence of a transitive example.
In 2000, Glasner and Weiss \cite{GW3}
 introduced the notion of {\it local equicontinuity}.
Let $(X,f)$ be a homeomorphism on a compact metric space.
This is said to be {\it locally equicontinuous} if every $x \in X$ is an 
 equicontinuity point on the orbit closure of $x$ itself.
Blanchard and Huang \cite{BH} announced the existence of
 a number of examples that are transitive, locally equicontinuous, and
 completely scrambled.
In this paper, we construct another set of examples that are 0-dimensional,
 transitive, and completely scrambled, but not locally equicontinuous.
We shall make use of the inverse limit of sequences of graph covers.

\section{Preliminaries}
In this section, we repeat the construction of graph covers for 0-dimensional systems originally given in \S 3 of \cite{Shimomura4}.
We also describe some notation for later use.
A pair $G = (V,E)$ consisting of a finite set $V$ and a relation $E \subseteq V \times V$ on $V$ can be considered as a directed graph with vertices $V$ and an edge from $u$ to $v$ when $(u,v) \in E$.
We assume that $G$ is edge surjective, i.e., for every vertex $v \in V$ there exist edges $(u_1,v),(v,u_2) \in E$.
Let $G_i = (V_i,E_i)$ with $i = 1,2$ be directed graphs.
A map $\fai : V_1 \to V_2$ is said to be a graph homomorphism
 if every edge is mapped to an edge; we describe this as $\fai : G_1 \to G_2$.
Suppose that a graph homomorphism $\fai : G_1 \to G_2$ satisfies the following condition:
\[(u,v),(u,v') \in E_1 \text{ implies that } \fai(v) = \fai(v').\]
In this case, $\fai$ is said to be {\it \pdirectional}.
Suppose that a graph homomorphism $\fai$ satisfies both of the following conditions:
\[(u,v),(u,v') \in E_1 \text{ implies that } \fai(v) = \fai(v') \myand \]
\[(u,v),(u',v) \in E_1 \text{ implies that } \fai(u) = \fai(u').\]
Then, $\fai$ is said to be {\it \bidirectional}.\\

%
%
\begin{defn}\label{defn:cover}
A graph homomorphism $\fai : G_1 \to G_2$ is called a {\it cover}\/ if it is a \pdirectional\ \tesgh.
\end{defn}
Let $\Gcal$ be a sequence $G_1 \getsby{\fai_1} G_2 \getsby{\fai_2} \dotsb$ of graph homomorphisms.
\begin{nota}
For $m > n$, let $\fai_{m,n} := \fai_{n} \circ \fai_{n+1} \circ \dotsb \circ \fai_{m-1}$.
\end{nota}
Then, $\fai_{m,n}$ is a graph homomorphism.
If all ${\fai_i}$ $(i \beposint)$ are edge surjective, then every $\fai_{m,n}$ is edge surjective.
Similarly, if all ${\fai_i}$ $(i \beposint)$ are covers, every $\fai_{m,n}$ is a cover.
%
%
Let us write $G_i = (V_i,E_i)$ for $i \benonne$.
Define
\[V_{\Gcal} := \seb (x_0,x_1,x_2,\dotsc) \in \prod_{i = 0}^{\infty}V_i~|~x_i = \fai_i(x_{i+1}) \text{ for all } i \benonne \sen \text{ and}\]
\[E_{\Gcal} := \seb (x,y) \in V_{\Gcal} \times V_{\Gcal}~|~(x_i,y_i) \in E_i \text{ for all } i \benonne\sen,\]
each equipped with the product topology.
\begin{nota}
For each $n \benonne$, the projection from $V_{\Gcal}$ to $V_n$ is denoted by $\fai_{\infty,n}$. 
\end{nota}
\begin{nota}\label{nota:decomp}
Let $X$ be a compact metrizable 0-dimensional space.
A finite partition of $X$ by non-empty clopen sets is called a {\it decomposition}.
The set of all decompositions of $X$ is denoted by $\Decomp(X)$.
Each $\Ucal \in \Decomp(X)$ is endowed with the discrete topology.
\end{nota}
\begin{nota}
Let $f : X \to X$ be a continuous surjective mapping from a compact metrizable 0-dimensional space $X$ onto itself.
Let $\Ucal$ be a decomposition of $X$.
Then, a map $\kappa_{\Ucal} : X \to \Ucal$ is defined as $\kappa_{\Ucal}(x) = U \in \Ucal$ if $x \in U \in \Ucal$.
A surjective relation $\funcdecomp{f}{\Ucal}$ on $\Ucal$ is defined as:
\[ \funcdecomp{f}{\Ucal} := \seb (u,v) ~|~ f(u) \cap v \nekuu \sen. \]
If the surjective map $f : X \to X$ is regarded as a subset $f \subseteq X \times X$ and $\Ucal$ is a decomposition of $X$, then $\funcdecomp{f}{\Ucal} = (\kappa_{\Ucal} \times \kappa_{\Ucal})(f)$.
In general, $(\Ucal,\funcdecomp{f}{\Ucal})$ is a graph, because $f$ is a surjective relation.
\end{nota}
%
%
We can state the following:
\begin{lem}\label{lem:invlimto0}
Let $\Gcal$ be a sequence $G_0 \getsby{\fai_0} G_1 \getsby{\fai_1} G_2 \getsby{\fai_2} \dotsb$ of covers.
Then, $V_{\Gcal}$ is a compact metrizable 0-dimensional space, and the relation $E_{\Gcal}$ determines a continuous  mapping from $V_{\Gcal}$ onto itself.
In addition, if the sequence is \bidirectional, then the relation $E_{\Gcal}$ determines a homeomorphism.
\end{lem}
\begin{proof}
See Lemma 3.5 of \cite{Shimomura4}.
\end{proof}
%
%
Let $\Gcal$ be a sequence $G_0 \getsby{\fai_0} G_1 \getsby{\fai_1} G_2 \getsby{\fai_2} \dotsb$ of covers.
Then, by the above lemma, $E_{\Gcal}$ defines a continuous surjective mapping from $V_{\Gcal}$ onto itself.
The 0-dimensional system $(V_{\Gcal},E_{\Gcal})$ is called the inverse limit of $\Gcal$, and is denoted by $G_{\infty}$.
We write $G_{\infty} = (V_{\Gcal},E_{\Gcal}) = (X,f)$.
%
%
\begin{nota}\label{nota:graph=covering-graph}
Let $\Gcal$ be a sequence $G_0 \getsby{\fai_0} G_1 \getsby{\fai_1} G_2 \getsby{\fai_2} \dotsb$ of covers.
Let $G_i = (V_i, E_i)$ for $i \benonne$.
For each $i \benonne$, we define:
\[\Ucal_{i} := \seb \fai_{\infty,i}^{-1}(u)~|~u \in V_i \sen,\]
which we can identify with $V_i$ itself.
\end{nota}
%
%
From \cite{Shimomura4}, we have the following:
%
%
\begin{thm}\label{thm:structure}
A topological dynamical system is 0-dimensional if and only if it is topologically conjugate to $G_{\infty}$ for some sequence of covers $G_0 \getsby{\fai_0} G_1 \getsby{\fai_1} G_2 \getsby{\fai_2} \dotsb$.
In addition, if all of the covers are bidirectional, then the resulting 0-dimensional system is a homeomorphism.
\end{thm}
%
\noindent We now give some notation that will be used later in the paper.
\itemb
\item[(N-1)]  We write $G_{\infty} = (X,f)$,
\item[(N-2)]  we fix a metric $d$ on $X$,
\item[(N-3)]  for each $i \bni$, we write $G_i = (V_i,E_i)$,
\item[(N-4)]  for each $i \bni$, we define $U(v) := \fai^{-1}_{\infty,i}(v)$ for $v \in V_i$ and $\Ucal_i := \seb U(v) ~|~v \in V_i\sen \in \Decomp(X)$, and 
\item[(N-5)]  for each $i \bni$, there exists a bijective map $V_i \ni v \leftrightarrow U(v) \in \Ucal_i$.
By this bijection, we obtain a graph isomorphism $G_i \cong (\Ucal_i,\funcdecomp{f}{\Ucal_i})$.
\itemn
Let $G = (V,E)$ be a surjective directed graph.
A sequence of vertices $(v_0,v_1,\dotsc,v_l)$ of $G$ is said to be a {\it walk}\/ of length $l$ if $(v_i, v_{i+1}) \in E$ for all $0 \le i < l$.
We denote  $l(w) := l$ and $V(w) := \seb v_0,v_1,\dotsc,v_l \sen$.
We say that a walk $w = (v_0,v_1,\dotsc,v_l)$ is a {\it path}\/
 if $v_i$ $(0 \le i \le l)$ are mutually distinct.
A walk $c = (v_0,v_1,\dotsc,v_l)$ is said to be a {\it cycle}\/ of period $l$
 if $v_0 = v_l$,
 and a cycle $c = (v_0,v_1,\dotsc,v_l)$
 is a {\it circuit} of period $l$ if the $v_i$ $(0 \le i < l)$ are mutually distinct.
Let $w_1 = (u_0,u_1,\dotsc, u_l)$ and $w_2 = (v_0,v_1,\dotsc, v_{l'})$ be walks
 such that $u_l = v_0$.
Then, we denote $w_1 + w_2 := (u_0,u_1,\dots,u_l,v_1,v_2,\dotsc,v_{l'})$.
Note that $l(w_1 + w_2) = l+l'$.
\section{Construction}
Let $G_0$ be a singleton graph
 with only one vertex $v_{0,0}$ and one edge $(v_{0,0},v_{0,0})$.
We shall construct a sequence of graph covers $G_0 \getsby{\fai_0} G_1 \getsby{\fai_1} G_2 \getsby{\fai_2} \dotsb$
 such that every $G_n$ $(n \ge 1)$ is a generalized figure-8 with a special vertex $v_{n,0} \in V_n$
 and a special edge $e_{n,0} = (v_{n,0},v_{n,0}) \in E_n$ for each $n \bni$.
We assume that $\fai_n(v_{n+1,0}) = v_{n,0}$ for each $n \bni$. 
 Thus, $\fai_n(e_{n+1,0}) = e_{n,0}$ for each $n \bni$.
With this setting, we construct a class of examples.
We assume that, for each $n \ge 1$,
 $G_n$ consists of circuits
 $\seb c_{n,1}, c_{n,2}, \dotsc, c_{n,n},e_{n,0} \sen$ such that
 for each $1 \le i < j \le n$, $V(c_{n,i}) \cap V(c_{n,j}) = \seb v_{n,0} \sen$.
We write $l(n,l) := l(c_{n,l})$ and $c_{n,l} =
 (v_{n,l,0} = v_{n,0},v_{n,l,1},v_{n,l,2},\dotsc, v_{n,l,l(n,l)} = v_{n,0})$.
Let us construct a cover.
We assume that $\fai_n(V(c_{n+1,n+1})) = v_{n,0}$ for each $n \bni$,
and that, for each $1 \le i \le n$,
 $\fai_n(c_{n+1,i}) = e_{n,0} + 2 c_{n,i}+2 c_{n,i+1} + \dotsb + 2c_{n,n} + e_{n,0}$.
These are bidirectional covers,
 and the resulting continuous surjection $(X,f)$ is a homeomorphism.
The length of each $c_{n+1,i}$ with $1 \le i \le n$ is determined by the length of
 each $c_{n,j}$ with $1 \le j \le n$,
 and we can take $l(c_{n+1,n+1}) > 1$ arbitrarily.
If we include  $l(c_{n+1,n+1}) = 1$, we cannot distinguish 
 $c_{n+1,n+1}$ with $e_{n+1,0}$.
Therefore, we avoid this case.
As stated in the previous section,
 $G_{\infty}$ is written as $(X,f)$, and we now use the notation described earlier.
It is clear that $X$ has no isolated points.
Thus, $(X,f)$ is a Cantor system.
\begin{thm}\label{thm:main}
The Cantor system $(X,f)$ is completely scrambled,
 topologically transitive, 
 and is not locally equicontinuous.
\end{thm}
\begin{nota}
We denote $p := (v_{1,0},v_{2,0},v_{3,0},\dotsc) \in X$,
 and every $e_{n,0}$ ($n \bni$) is simply described  as $e$ if there no possibility of confusion.
\end{nota}
It is obvious that $p$ is a fixed point.
From the construction of the covers, the next lemma follows easily:
\begin{lem}
For $m > n$ and $1 \le l \le l' \le n$, it follows that $l(m,l)>l(n,l')$.
\end{lem}

\begin{lem}\label{lem:no-periodic-orbits}
We have that $p \in X$ is the only fixed point, and
 this is the only periodic point.
\end{lem}
\begin{proof}
The first statement is obvious.
Suppose that there exists a periodic point other than $p$.
Then, there exists some $N > 0$ such that, for all $n \ge N$,
 there exists a circuit $c_n$ of $G_n$ of the same period,
 and the $\fai_n$ are isomorphisms of these circuits.
This contradicts the construction of this graph cover.
\end{proof}
In this section, we show that $(X,f)$ is completely scrambled
 by stating successive lemmas.
Broadly, we show, in the following order, that:
\[
\begin{array}{l}
\liminf_{i \to \pinfty}d(f^k(x),f^k(p)) = 0 \text{ for } x \ne p,\\
\limsup_{i \to \pinfty}d(f^k(x),f^k(p)) > 0 \text{ for } x \ne p,\\
\limsup_{i \to \pinfty}d(f^k(x),f^k(y)) = 0
 \text{ for } x \ne y \in X\setminus \seb p \sen, \text{ and}\\
\limsup_{i \to \pinfty}d(f^k(x),f^k(y)) > 0
 \text{ for } x \ne y \in X\setminus \seb p \sen \text{ in separate cases}.
\end{array}
\]
\begin{lem}\label{lem:proximal-to-p}
For each $x \in X$,
 it follows that $\liminf_{k \to \pinfty}d(f^k(x),p) = 0$.
\end{lem}
\begin{proof}
For any $n \bni$, it follows that the sequence $\fai_{\infty,n}(f^k(x))$ $(k > 0)$ follows a walk
 of $G_n$.
It is clear that this walk passes $v_{n,0}$ infinitely many times.
Thus, the conclusion is evident.
\end{proof}
%
%
%
%
\begin{nota}
For $v_{n,i,j} \in V_n$ with $0 < j < l(n,i)$, we denote 
 $\remn(v) := l(n,i) - j$.
For an $x \in U(v) \subset X$, $\remn(v)$ steps remain
 until we reach $U(v_{n,0})$,
 i.e., $f^{i}(x) \notin U(v_{n,0})$ for $0 \le i < \remn(v)$
 and $f^{\remn(v)}(x) \in U(v_{n,0})$.
\end{nota}
\begin{nota}
For $v \in V_n$, we denote the {\it degree}\/ of $v$ as follows:
\[ \deg(v) =
\begin{cases}
\pinfty & \text{ if } v = v_{n,0}\\
i & \text{ if } v \in V(c_{n,i})\setminus \seb v_{n,0} \sen
\end{cases}
\]
\end{nota}
\begin{lem}\label{lem:degree-non-increasing}
Let $x = (v_0,v_1,v_2,\dotsc) \in X$.
For $n < n'$, it follows that $\deg v_n \ge \deg v_{n'} \ge 1$.
\end{lem}
\begin{proof}
By the construction of $\fai_n$ $(n \bni)$, the proof is evident.
\end{proof}
\begin{nota}\label{nota:degree-of-each-x}
Let $x = (v_0,v_1,v_2,\dotsc) \in X$.
Then, we define the {\it degree}\/ of $x$ as
 $\deg x := \min \seb \deg v_i \mid i \bpi \sen$.
Note that $\deg x = \pinfty$ implies that $x = p$.
\end{nota}
\begin{lem}\label{lem:degree-constant-upper}
For $p \neq x = \pstrzinf{u} \in X$,
 there exists an $N > 0$ such that $\deg u_N = \deg u_{N+1} = \dotsb = \deg x$.
\end{lem}
\begin{proof}
The proof is obvious.
\end{proof}
The next lemma is not explicitly used in this paper, but we believe it helps to clarify our argument.
\begin{lem}\label{lem:constant-degree}
The degree of each orbit is constant.
\end{lem}
\begin{proof}
Let $x = (v_0,v_1,v_2,\dotsc) \in X$.
We must show that $\deg f(x) = \deg x$.
If $x = p$, then the conclusion is obvious.
Let $x \ne p$ and $\deg(x) = l$.
Then, there exists $n \bni$ such that $v_n \ne v_{n,0}$.
By Lemma \ref{lem:degree-constant-upper}, we can assume that
 $\deg v_n = \deg v_{n+1} = \dotsb = l$.
Then, for $k \ge 0$, we get $v_{n+k} \in c_{n+k,l}$.
Thus, for $k \ge 1$, it follows that
 $\remn(v_{n+k}) >
 2l(c_{n+k-1,l+1})+2l(c_{n+k-1,l+2}) + \dotsb + 2l(c_{n+k-1,n+k-1})+1$.
Thus, $\remn(v_{n+k}) \to \pinfty$ as $k \to \pinfty$.
Let $(v_{n+k},v'_{n+k})$ be an edge of the circuit $c_{n+k,l}$.
For $k \ge 1$, it follows that $v'_{n+k} \ne v_{n+k,0}$.
Because $f(x) = (\dotsc, v'_{n+k},v'_{n+k+1},v'_{n+k+2},\dotsc)$,
 it is clear that $\deg f(x) = l$.
\end{proof}
%
%
\begin{lem}\label{lem:x-p-limsup-positive}
For $x \ne p$,
 it follows that $\limsup_{k \to \pinfty}d(f^k(x),p) > 0$.
\end{lem}
\begin{proof}
Let $\deg x = l$ and $x = (u_0,u_1,u_2,\dotsc)$.
Because $p \ne x$, we have that $l < \pinfty$.
By Lemma \ref{lem:degree-constant-upper},
 there exists an $N \bni$ such that $\deg u_{n} = l$ for all $n \ge N$.
Let $n \ge N$.
Then, $u_{n} \in c_{n,l}$.
Let $p_n$ be the path of $c_{n,l}$ from $u_{n}$ to $v_{n,0}$.
By the definition of $N$, it follows that
 $\fai_{n,N}(u_{n}) \in V(c_{N,l})$.
For $0 \le i \le \remn(u_{n})$,
 $\fai_{\infty,n}(f^i(x))$ follows the path from $u_{n}$ to $v_{n,0}$.
We get $\fai_{n-1}(c_{n,l})
 = e + 2c_{n-1,l} + 2c_{n-1,l+1} + \dotsb + 2c_{n-1,n-1} + e$.
Thus, $p_n$ follows the totality of $2c_{n-1,l+1}$,
and $\fai_{n-1,N}(2c_{n-1,l+1})$ winds around $c_{N,l+1}$ exactly $2^n$ times.
Therefore, $\fai_{n,N}(p_n)$ winds around $c_{N,l+1}$ at least $2^n$ times.
Fixing $\tau \in c_{N,l+1}$ such that $\tau \ne v_{N,0}$,
 $\fai_{\infty,N}(f^i(x)) = \tau$ at least $2^n$ times.
Because $n > 0$ is arbitrary, the conclusion is now obvious.
\end{proof}
\begin{lem}\label{lem:xy-proximal}
Let $x \ne y$ be distinct from $p$.
Then, it follows that 
\[ \liminf_{k \to \pinfty}d(f^k(x),f^k(y)) = 0.\]
\end{lem}
\begin{proof}
Let $x \ne y \in X$ be distinct from $p$, $x = \pstrzinf{u}$, and $y = \pstrzinf{v}$.
It follows that $\deg x, \deg y < \pinfty$.
Let $l = \deg x$ and $l' = \deg y$.
By Lemma \ref{lem:degree-constant-upper}, there exists $N > 0$ such that
 $\deg u_N = \deg u_{N+1} = \dotsb = l$ and $\deg v_N = \deg v_{N+1} = \dotsb = l'$.
Note that $l,l' \le N$.
For all $n > N$, $u_{n} \in c_{n,l}$ and $v_{n} \in c_{n,l'}$.
Because
\[
\begin{array}{lll}\fai_{n-1}(c_{n,l}) & = & e_{n-1,0} +
 2c_{n-1,l} + 2c_{n-1,l+1} + 2c_{n-1,l + 2}\\
 & & + \dotsb + 2c_{n-1,n-1} + e_{n-1,0},
\end{array}
\]
it follows that
\[
\begin{array}{lll}\fai_{n,N}(c_{n,l}) & = & e_{N,0} + 2\fai_{n-1,N}(c_{n-1,l})
 + 2\fai_{n-1,N}(c_{n-1,l+1})
 + 2\fai_{n-1,N}(c_{n-1,l + 2})\\
 & & + \dotsb + 2\fai_{n-1,N}(c_{n-1,N})\\
 & & + 2\fai_{n-1,N}(c_{n-1,N+1})+ \dotsb + 2\fai_{n,N}(c_{n-1,n-1}) + e_{N,0}.
\end{array}
 \]
We write the first two lines as 
\[
\begin{array}{lll} p_n & = & e_{N,0} + 2\fai_{n-1,N}(c_{n-1,l}) + 2\fai_{n-1,N}(c_{n-1,l+1})
 + 2\fai_{n-1,N}(c_{n-1,l + 2})\\
 & & + \dotsb + 2\fai_{n-1,N}(c_{n-1,N})
\end{array}
 \]
and the last line as 
 $q_n = 2\fai_{n-1,N}(c_{n-1,N+1})+ \dotsb + 2\fai_{n,N}(c_{n-1,n-1}) + e_{N,0}$.
Note that we can write $q_n = L(N,n)e_{N,0}$ for the positive integer $L(N,n)$.
It is clear that $L(N,n) \to \pinfty$ as $n \to \pinfty$.
Because $\fai_{\infty,N}(x) \in V(c_{N,l})\setminus \seb v_{N,0} \sen$,
 the sequence $\fai_{\infty,N}(f^i(x))$ $(i \ge 0)$
 starts from within $2\fai_{n-1,N}(c_{n-1,l})$.
Therefore, this sequence lies in the $p_n$ for small $i$,
 and enters into $q_n$ for larger $i$,
 eventually reaching the end of $q_n$.
Note that $p_n$ is a repetition of $e_0$ and $c_{N,j}$ with $l \le j \le N$.
Let $M(N) := \max_{1 \le j \le N} l(c_{N,j}) < \pinfty$.
Take $n$ to be sufficiently large such that $M(N) < L(N,n)$.
The same situation occurs for $y$, and at least one of
 $\fai_{\infty,N}(f^i(x))$ or $\fai_{\infty,N}(f^i(y))$
 enters into $q_n$;
 in the period that is less than or equal to $M(N)$, the other takes the value $v_{N,0}$.
Let $L(n) := \min \seb \remn(\fai_{\infty,n}(x)), \remn(\fai_{\infty,n}(y) \sen$.
Then, there exists an $i$ such that
 both $L(n) - M(N) \le i \le L(n)$ and $\seb f^i(x), f^i(y) \sen \subset U(v_{N,0})$
 are satisfied.
Because $L(n) \to \pinfty$ as $n \to \pinfty$,
 we get $\liminf_{i \to \pinfty}d(f^i(x),f^i(y)) \le \diam U(v_{N,0})$.
Because we can take $N$ to be arbitrarily large,
 we have $\liminf_{i \to \pinfty}d(f^i(x),f^i(y)) = 0$.
\end{proof}
\begin{nota}
Let $x \ne y \in X$, $x = \pstrzinf{u}$, and $y = \pstrzinf{v}$.
Let $n \bni$.
Suppose that there exists some $1 \le i \le n$ such that $u_n,v_n \in V(c_{n,i})$,
 $u_n = v_{n,i,j}$, and $v_n = v_{n,i,j'}$.
Then, we denote $\gap (u_n,v_n) := j' - j$.
\end{nota}
\begin{lem}\label{lem:bounded-gap-on-the-same-orbit}
Let $x \ne y \in X$, $x = \pstrzinf{u}$, and $y = \pstrzinf{v}$.
Suppose that $\deg x = \deg y < \pinfty$
 and $\limsup_{n \to \pinfty} \left|\gap(u_n,v_n)\right| < \pinfty$.
Then, there exists a $d \in \Z$ such that $f^d(x) = y$.
\end{lem}
\begin{proof}
Because there exists an integer $d$ and a subsequence $n_k~(k \bpi)$ such that
 $\gap(u_{n_k},v_{n_k}) =d $ for all $k$, the conclusion is obvious.
\end{proof}
\begin{lem}\label{lem:on-the-same-orbit-positive}
Let $x \ne y \in X$ be such that $f^d(x) = y$ for some $d \ne 0$.
Then, it follows that
 \[\limsup_{k \to \pinfty}d(f^k(x),f^k(y)) > 0.\]
\end{lem}
\begin{proof}
We prove the statement by contradiction.
Assume that $\limsup_{k \to \pinfty}d(f^k(x),f^k(y)) = 0$.
By Lemma \ref{lem:x-p-limsup-positive}, $\limsup_{k \to \pinfty}d(p,f^k(x)) > 0$.
Thus, there exists a point $z \ne p$ and a subsequence $k_i$ $(i \bpi)$ such that
  $f^{k_i}(x) \to z$ as $i \to \pinfty$.
By the assumption, we have $f^{k_i}(y) \to z$ as $i \to \pinfty$.
Thus, we get $f^d(z) = z$.
Because $z \ne p$, this contradicts Lemma \ref{lem:no-periodic-orbits}.
\end{proof}
\begin{nota}\label{nota:more-than-l+1}
Let $m > d \ge 1$.
For $n > m$, we denote 
\[r_{n,d,m} := 2\fai_{n,m}(c_{n,d}) + 2\fai_{n,m}(c_{n,d+1})
               + \dotsb+ 2\fai_{n,m}(c_{n,n}) + e_{m,0}.\]
Note that in $r_{n,d,m}$, there is no occurrence of $c_{m,d'}$ with $d' < d$.
\end{nota}

\begin{lem}\label{lem:gap-calculation}
Let $n \gg N \gg  l \ge 1$.
In $\fai_{n+2,N}(c_{n+2,l})$,
 let $c_{N,l} + s + c_{N,l}$ be two occurrences
 of $c_{N,l}$ with no occurrence of $c_{N,l}$ in $s$.
Then, $l(s) = m - N + 1 + \sum_{k = N}^{m}l(r_{k,l+1,N})$ for some $m$ with
 $N \le m \le n$.
Besides such appearances, we have those of the form $c_{N,l}+c_{N,l}$.
Further, when we write $\fai_{n+1,N}(c_{n+1,l}) = \dotsb + c_{N,l} + s$ with no occurrence of $c_{N,l}$ in walk $s$, it follows that $l(s) = \sum_{k = N}^{n}r_{k,l+1,N}$.
\end{lem}
\begin{proof}
We abbreviate  $e_{m,0}$ as $e$ for all $m \bni$.
It follows that
\[\fai_{n+1}(c_{n+2,l}) = e + 2c_{n+1,l} + 2c_{n+1,l+1} + \dotsb + 2c_{n+1,n+1} + e.\]
Thus, we obtain the following:
\[\fai_{n+1}(c_{n+2,l}) = e + 2 c_{n+1,l} + r_{n+1,l+1,n+1}.\]
Therefore,
 it is sufficient to consider the gap between occurrences of $c_{N,l}$ in $2c_{n+1,l}$.
We shall calculate the largest gap in $2c_{n+1,l} = c_{n+1,l} + c_{n+1,l}$,
 and show that it is between the last occurrence of $c_{N,l}$ in the first $c_{n+1,l}$
 and the first occurrence of it in the last $c_{n+1,l}$.
We calculate 
\[\fai_{n+2,n}(c_{n+2,l}) = e + 2(e + 2c_{n,l} + r_{n,l+1,n}) + r_{n+1,l+1,n}.\]
In the last expression,
 both $2c_{n,l}$ and $c_{n,l} + r_{n,l+1,n} +e + c_{n,l}$ occur.
Thus, the largest gap is in $c_{n,l} + r_{n,l+1,n} +e + c_{n,l}$,
 between the last $c_{N,l}$ in $c_{n,l}$ and the first one in $c_{n,l}$, 
 as stated above.
This gives
\[
\begin{array}{lll}
\fai_{n-1}(c_{n,l} + r_{n,l+1,n} +e + c_{n,l})
 & = &  e + 2c_{n-1,l} + r_{n-1,l+1,n-1} + r_{n,l+1,n-1} + e\\
 & &  + e +2c_{n-1,l} + r_{n-1,l+1,n-1}.
\end{array}
\]
Thus, by induction, if we project the above expression by $\fai_{n-1,N}$,
 the last occurrence of $c_{N,l}$ in the first $2c_{n-1,l}$
 and the first occurrence of $c_{N,l}$ in the last occurrence of $2c_{n-1,l}$
 can be bridged as:
\[ c_{N,l} + r_{N,l+1,N}+r_{N+1,l+t,N} + \dotsb +r_{n,l+1,N} + (n -N +1)e + c_{N,l}.\]
Therefore,
 if we write $s = r_{N,l+1,N}+r_{N+1,l+t,N} + \dotsb +r_{n,l+1,N} + (n -N +1)e$,
 then $l(s) = n - N + 1+ \sum_{k = N}^{n}l(r_{k,l+1,N})$.
This also shows that it is the largest gap in $2c_{n+1,l}$.
We have seen that
 the gap between occurrences of $c_{N,l}$ appears as the largest gap in 
 $2c_{m,l}$ with $N < m \le n+1$.
Beside these gaps, of course there exist occurrences of the form $2c_{N,l}$.
It remains to demonstrate the last statement.
As in the above calculation,
 it follows that $\fai_n(c_{n+1,l}) = e + c_{n,l} + c_{n,l} + r_{n,l+1,n}$.
Consequently,
 we have
\[
\begin{array}{lll}
\fai_{n+1,n-1}(c_{n+1,l}) & = & \dotsb + \fai_{n-1}(c_{n,l}) + r_{n,l+1,n-1}\\
 &  = &  \dotsb + e + 2c_{n-1,l} + r_{n-1,l+1,n-1} + r_{n,l+1,n-1}.
\end{array}
\]
In this way, we get $\fai_{n+1,N}(c_{n+1,l})
 = \dotsb + c_{N,l} + r_{N,l+1,N} + r_{N+1,l+1,N} + \dotsb + r_{n,l+1,N}$.
Thus, we have the desired conclusion.
\end{proof}
From the proof of the above lemma, we get the following:
\begin{lem}\label{lem:gap-between-CnCn}
In $\fai_{n,N}(c_{n,l} + c_{n,l})$,
 let $A$ be the last occurrence of $c_{N,l}$ in the first $c_{n,l}$,
 and $B$ the first occurrence in the last $c_{n,l}$.
If we write $\fai_{n,N}(c_{n,l}+c_{n,l}) = \dotsb + A + s + B + \dotsb$,
then 
\[ s = r_{N,l+1,N}+r_{N+1,l+t,N} + \dotsb +r_{n-1,l+1,N} + (n -N)e.\]
In particular, $l(s) = n - N + \sum_{k = N}^{n-1}l(r_{k,l+1,N})$.
\end{lem}
\begin{nota}\label{nota:gap}
We denote $g_{n,l,N} = n - N + \sum_{k = N}^{n-1}l(r_{k,l+1,N})$.
Then, $g_{n,l,N}$ is the largest gap between occurrences of $c_{N,l}$ in $2c_{n,l}$.
\end{nota}
From this point, for $x \ne y$ with $x,y \in X\setminus \seb p \sen$, we present successive lemmas to check that\\ 
 $\limsup_{k \to \pinfty}d(f^k(x) ,f^k(y)) > 0$.

\begin{lem}
Let $x \ne y \in X$ be such that $\deg x = \deg y < \pinfty$.
Then, $\limsup_{k \to \pinfty}d(f^k(x),f^k(y)) > 0$.
\end{lem}
\begin{proof}
Let $x \ne y \in X$ be such that $\deg x = \deg y < \pinfty$.
Let $\deg x = \deg y = l$,
 $x = (u_0,u_1,u_2,\dotsc)$, and $y = (v_0,v_1,v_2,\dotsc)$.
By Lemma \ref{lem:degree-constant-upper},
 there exists an $N \bni$ such that $\deg u_{n} = \deg v_{n}$ for all $n \ge N$.
We can take $N$ such that $N \gg l$,
and a $\gap(u_n,v_n)$ is defined for every $n \ge N$.
Suppose that $\limsup_{n \to \pinfty}\left|\gap(u_n,v_n)\right| < \pinfty$.
Then,
 by Lemma \ref{lem:bounded-gap-on-the-same-orbit}
 and Lemma \ref{lem:on-the-same-orbit-positive},
 it follows that $\limsup_{k \to \pinfty}d(f^k(x),f^k(y)) > 0$.
Therefore, we assume that
 $\limsup_{n \to \pinfty} \left|\gap(u_n,v_n)\right| = \pinfty$.
Broadly, we shall show that
 one of the two orbits enters a domain of degree larger than $l$,
 and, after a long time, another orbit still takes the vertices of degree $l$.
Let $n \gg N$.
By the definition of $N$, it follows that
 $\fai_{n,N}(u_{n}) \in V(c_{N,l})$
 and $\fai_{n,N}(v_{n}) \in V(c_{N,l})$.
By Lemma \ref{lem:gap-calculation}, both $\remn(u_n) \to \pinfty$ and 
 $\remn(v_n) \to \pinfty$ hold.
Let $K(n) = \min \seb \remn(u_n), \remn(v_n) \sen$.
For $0 \le i \le K(n)$,
 both $\fai_{\infty,n}(f^i(x))$ and $\fai_{\infty,n}(f^i(y))$ follow the path on $c_{n,l}$ until one of them reaches the end.
Without loss of generality, we assume that $\gap(u_n,v_n) > 0$ for infinitely many $n$,
 and we assume that we can take an arbitrarily large $n$
 with arbitrarily large $\gap(u_n,v_n)>0$.
Thus, $\fai_{\infty,n}(f^i(y))$ follows the last $c_{N,l}$ first.
To catch the timing of this last $c_{N,l}$, we take an $L(n)  > 0$ such that
 $\deg(\fai_{\infty,n}(f^{L(n)-1}(y))) = l$ and, for $L(n) \le i \le K(n)$,
 $\deg(\fai_{\infty,n}(f^i(y))) \ge l+1$.
Let A(n) := $\sum_{k = N}^{n-1}l(r_{k,l+1,N}) = K(n) - L(n)$.
By Lemma \ref{lem:gap-calculation}, the gap between occurrences of $c_{N,l}$
 is at most $B(n) := \sum_{k = N}^{n-2}r_{k,l+1,N}$.
Therefore, $\fai_{\infty,N}(f^i(x))$ follow $c_{N,l}$ for at least one $i$ 
 with $L(n) \le i \le K(n)$.
We have to show that we can take $i$ with $L(n) \le i \le K(n)$
 such that $\deg \fai_{\infty,N}(f^i(x)) = l$
 is arbitrarily large.
We have:
\[
\begin{array}{lll}
A(n) - B(n) & = & \sum_{k = N}^{n-1}l(r_{k,l+1,N})-\sum_{k = N}^{n-2}r_{k,l+1,N}\\
 & = & l(r_{n-1,l+1,N})\\
 & \to & \pinfty \text{ as } n \to \pinfty.
\end{array}
\]
Let $L(n) \le  i_n \le K(n)$ be the largest $i$
 with $L(n) \le  i \le K(n)$ and $\deg \fai_{\infty,N}(f^i(x)) = l$.
Then, it follows that $i_n + B(n) > K(n) > A(n)$.
Thus, $i_n \ge A(n) - B(n)$ is unbounded as $n \to \pinfty$.
Therefore, $\deg \fai_{\infty,N}(f^i(x)) \ne \deg \fai_{\infty,N}(f^i(y))$
 for infinitely large $i > 0$.
This concludes the proof.
\end{proof}
By Notation \ref{nota:degree-of-each-x}, for $x \ne p$,
 we have $\deg x < \pinfty$.
\begin{lem}
Let $x \ne y \in X$ be distinct from $p$ and $\deg x +2 \le \deg y < \pinfty$.
Then, it follows that 
\[\limsup_{k \to \pinfty}d(f^k(x),f^k(y)) > 0.\]
\end{lem}
\begin{proof}
Let $x = (u_0,u_1,u_2,\dotsc)$ and $y = (v_0,v_1,v_2,\dotsc)$.
Let $\deg x = l$ and $\deg y = l'$.
Then, it follows that $l+2 \le l'$.
As before, fix a large $N > 0$ such that
 $\deg u_n = l$ and $\deg v_n = l'$ for all $n \ge N$.
First, we show that the sequence $\fai_{\infty,N}(f^i(x))$ with $i \ge 0$
 treads $c_{N,l+1}$ infinitely many times.
For each $n > N$, we get
\[ \fai_{n-1}(c_{n,l}) = e + 2c_{n-1,l} + 2c_{n-1,l+1} + r_{n-1,l+2,n-1}.\]
In the above expression, for large enough $n$, $u_{n-1}$ lies in $2c_{n-1,l}$.
Therefore, $\fai_{\infty,N}(f^i(x))$ with $i \ge 0$ passes $2c_{n-1,l+1}$;
it follows that it passes $c_{N,l+1}$ at least $2^{n-N}$ times.
Because $n$ is arbitrarily large, it passes $c_{N,l+1}$ infinitely many times.
Next, note that $\fai_{\infty,N}(f^i(y))$ with $i \ge 0$ passes only
$e$ or $c_{N,m}$ with $l+1 < l' \le m \le N$.
The conclusion is obvious.
\end{proof}
As in Notation \ref{nota:gap}, in $\fai_{m,N}(2c_{m,l})$, the largest gap
 between occurrences of $c_{N,l}$ is calculated.
In the following lemmas, we also consider the pattern of {\it the occurrence of gaps}. 
\begin{lem}\label{lem:usual-gap-occurrence}
Let $n \gg N > l$.
In $\fai_{n,N}(c_{n,l})$,
 whenever there exist two occurrences of a gap in $c_{N,l}$ with length $g_{m,l,N}$,
 there exists a gap in $c_{N,l}$ of length $g_{m',l,N}$ between them, where $m' > m$.
\end{lem}
\begin{proof}
We first make the following calculation:
\[
\begin{array}{lll}
\fai_{n,n-2}(c_{n,l}) & = & \fai_{n-2}(e + 2c_{n-1,l} + r_{n-1,l+1,n+1})\\
 & = & e + 2\fai_{n-2}(c_{n-1,l}) + \fai_{n-2}(r_{n-1,l+1,n-1})\\
 & = & e + 2(e + 2c_{n-2,l} +r_{n-2,l+1,n-2}) + r_{n-1,l+1,n-2}\\
 & = & e + e + 2c_{n-2,l} +r_{n-2,l+1,n-2} + e + 2c_{n-2,l} +r_{n-2,l+1,n-2}\\
 &   & + r_{n-1,l+1,n-2}.
\end{array}
\]
Let us project the above expression by $\fai_{n-2,N}$.
Then, we find the occurrence of gaps as follows:
\[ \dotsb (\text{gap of } g_{n-2,l,N}) \dotsb
   (\text{gap of } g_{n-1,l,N}) \dotsb (\text{gap of } g_{n-2,l,N}) \dotsb.
\]
By an easy induction, we obtain the conclusion.
\end{proof}
\begin{lem}\label{lem:tail-gap-occurrence}
Let $n \gg N \gg l$.
In $\fai_{n,N}(c_{n,l})$, even after all occurrences of $c_{N,l}$,
 there exists an occurrence of $c_{N,l+1}$.
We write $\fai_{n,N}(c_{n,l}) = \dotsb + c_{N,l} + s$ when there is no occurrence of $c_{N,l}$
in $s$.
Then, in $s$, there exist two gaps in $c_{N,l+1}$ of length $g_{n-1,l+1,N}$
 such that all gaps in $c_{N,l+1}$ between them have lengths of  less than $g_{n-1,l+1,N}$.
Furthermore, if we take $n$ to be sufficiently large, after the last occurrence of $c_{N,l}$,
 there exists an arbitrarily large interval before the last occurrence of two such gaps
 of length $g_{n-1,l+1,N}$.
\end{lem}
\begin{proof}
We can calculate that:
\[
\begin{array}{lll}
\fai_{n,n-2}(c_{n,l}) & = & \fai_{n-2}(e + 2c_{n-1,l} + 2c_{n-1,l+1} + r_{n-1,l+2,n+1})\\
 & = & \dotsb +  \fai_{n-2}(c_{n-1,l})
              + 2\fai_{n-2}(c_{n-1,l+1}) + \fai_{n-2}(r_{n-1,l+2,n-1})\\
 & = & \dotsb + c_{n-2,l} + 2c_{n-2,l+1} + r_{n-2,l+2,n-2}\\
 &   & \hphantom{\dotsb} + 2(e + 2c_{n-2,l+1} + r_{n-2,l+2,n-2}) + r_{n-1,l+2,n-2}\\
 & = & \dotsb + c_{n-2,l} +2c_{n-2,l+1} + r_{n-2,l+2,n-2}\\
 &   & \hphantom{\dotsb} + e + 2c_{n-2,l+1} + r_{n-2,l+2,n-2}\\
 &   & \hphantom{\dotsb} + e + 2c_{n-2,l+1} + r_{n-2,l+2,n-2} + r_{n-1,l+2,n-2}\\
 & = & \dotsb + c_{n-2,l} + c_{n-2,l+1} \\
 & &  \hphantom{\dotsb}+ c_{n-2,l+1} + r_{n-2,l+2,n-2} + e + c_{n-2,l+1}
 \quad \cdots \cdots \quad \text{(1)}\\
 & &  \hphantom{\dotsb}+ c_{n-2,l+1} + r_{n-2,l+2,n-2} + e + c_{n-2,l+1}
 \quad \cdots \cdots \quad\text{(2)}\\
 & &  \hphantom{\dotsb}+ c_{n-2,l+1} + r_{n-2,l+2,n-2} + r_{n-1,l+2,n-2}
\end{array}
\]
We consider the projection by $\fai_{n-2,N}$ of the above expression.
Then, in lines (1) and (2),
 there exists a gap in $c_{N,l+1}$ of length $g_{n-1,l+1,N}$,
 and the lengths of the gaps in $c_{N,l+1}$ between them are at most $g_{n-2,l+1,N}$.
This concludes the first part of the claim.
Because $l(c_{n-2,l+1}) \to \pinfty$ as $n \to \pinfty$,
 the last claim is obvious from the above expression.
\end{proof}
\begin{lem}
Let $x \ne y \in X$ be distinct from $p$, and $\deg x +1 = \deg y < \pinfty$.
Then, it follows that
\[\limsup_{k \to \pinfty}d(f^k(x),f^k(y)) > 0.\]
\end{lem}
\begin{proof}
Let $x \ne y \in X$ be distinct from $p$.
Let $\deg x  = l$.
Then, $\deg y = l+1$.
Let $x = (u_0,u_1,u_2,\dotsc)$ and $y = (v_0,v_1,v_2,\dotsc)$.
As we have already shown, there exists an $N > 0$
 such that $\deg u_{n} = l$ and $\deg v_{n} = l+1$ for all $n \ge N$.
The sequence $\fai_{\infty,N}(f^i(x))$ with $i \ge 0$ passes through only
$e$ or $c_{N,m}$ with $l \le m \le N$,
 and the sequence $\fai_{\infty,N}(f^i(y))$ with $i \ge 0$ passes through only
 $e$ or $c_{N,m}$ with $l + 1 \le m \le N$.
Therefore, if $\fai_{\infty,N}(f^i(x))$ with $i \ge 0$
 passes $c_{N,l}$ infinitely many times, then the conclusion is obvious.
Therefore, we assume that $\fai_{\infty,N}(f^i(x))$ with $i \ge 0$
 treads $c_{N,l}$ only finitely many times.
Note that $\fai_{\infty,N}(f^i(x))$ with $i \ge 0$ passes $c_{N,l+1}$
 infinitely many times.
On the other hand, if $\fai_{\infty,N}(f^i(y))$ with $i \ge 0$ only take values 
 on $c_{N,l+1}$ a finite number of times, then the conclusion is again obvious.
Therefore, we assume that $\fai_{\infty,N}(f^i(y))$ with $i \ge 0$ treads
 on $c_{N,l+1}$ infinitely many times.
%
%
Because $\fai_{n,N}(c_{n,l})$ contains $c_{N,l}$, there is some fixed $i_0 \in \Z$
 for which $\fai_{\infty,N}(f^{i_0-1}(x))$ becomes the last passage on
 $c_{N,l}\setminus \seb v_{N,0} \sen$
 and $\fai_{\infty,N}(f^{i_0}(x)) = v_{N,0}$.
Therefore, shifting $x$ and $y$ by $f^{i_0}$,
 we assume that $\fai_{\infty,N}(f^i(x))$ with $i \ge 0$ does not pass $c_{N,l}$.
For arbitrarily large $K > 0$, taking a large $n > N$,
 $\fai_{\infty,n}$ projects the orbit $f^i(x)$ with $0 \le i \le K$
 onto a path of $c_{n,l}$.
Therefore, for every gap in $c_{N,l+1}$, there exists an $n > N$ such that
 the gap is seen in $\fai_{n,N}(c_{n,l})$.
By Lemma \ref{lem:tail-gap-occurrence}, as $\fai_{\infty,N}(f^i(x))$ $(i = 0,1,\dotsc)$
 proceeds, there exist a couple of gaps in $c_{N,l+1}$ of length $g_{n-1,l+1,N}$,
 between which no larger gaps occur.
Furthermore, this occurs for arbitrarily large $i > 0$ if $n$ is large enough.
On the other hand, for arbitrarily large $K > 0$, taking a large $n > N$,
 $\fai_{\infty,n}$ projects the orbit $f^i(y)$ with $0 \le i \le K$
 onto a path of $c_{n,l+1}$.
By Lemma \ref{lem:usual-gap-occurrence},
 if $\fai_{\infty,N}(f^i(y))$ with $i > 0$ encounters a couple of gaps in
 $c_{N,l+1}$ of length $g_{n-1,l+1,N}$, a gap in $c_{N,l+1}$ of larger length
 must exist between them. 
Therefore, we obtain the desired conclusion.
\end{proof}
By proving the lemma above,
 we have shown that $(X,f)$ is completely scrambled.
The next lemma proves that $(X,f)$ is topologically transitive.
\begin{lem}
There exists an $x_0$ such that $\seb f^i(x_0) ~|~i \bni \sen$ is dense in $X$.
\end{lem}
\begin{proof}
Fix an arbitrary $N > 0$.
In our notation, $c_{N,1} = (v_{N,1,0},v_{N,1,1},\dots,v_{N,1,l(N,1)} = v_{N,0})$.
Let $u_N = v_{N,1,1}$.
It follows that $\fai_{N-1}(u_N) = v_{N-1,0}$.
Because $\fai_n(c_{n+1,l}) = e + 2c_{n,l} + \dotsb$ for all $n > 0$,
 we get $\fai_N(v_{N+1,1,2}) = v_{N,1,1} = u_N$.
In this way, if $u_n$ is defined, then $u_{n+1}$ is defined as the
 first occurrence of $u \in V(c_{n+1,1})$ such that $\fai_n(u) = u_n$.
We define 
 $x_0 := (v_{0,0},v_{1,0},\dotsc,v_{N-1,0},u_N, u_{N+1}, \dots)$.
Let $n \gg N$ be arbitrarily large.
Then, $\fai_{\infty,n+1}(f^i(x_0))$ with $i > 0$ follows
 a path $(u_{n+1},\dots,v_{n+1,0})$ in $c_{n+1,1}$.
Because $\fai_n(c_{n+1,1})$ winds around $c_{n,1}$ twice,
 $\fai_{\infty,n+1}(f^i(x_0))$ with $i > 0$ passes all vertices of $c_{n,1}$.
It is obvious that $\fai_{n-1}(c_{n,1})$ passes all vertices of $G_{n-1}$.
Because $n$ is arbitrary, the conclusion is obvious.
\end{proof}
The following lemma shows that $(X,f)$ is not locally equicontinuous.
\begin{lem}\label{lem:does-not-have-equicontinuity-points}
Let $x \in X$.
Then,  for every sufficiently large $n > N > 0$, 
 we have some $v \in V(G_n)$ and $i_n \in \Z$ such that, for $y_n = f^{i_n}(x)$,
 it follows that $x,y_n \in U(v)$ and there exists an $i > 0$
 with $\fai_{\infty,N}(f^i(x)) \ne \fai_{\infty,N}(f^i(y_n))$.
\end{lem}
\begin{proof}
Let $x \ne p$.
Let $\deg x = l < \pinfty$, and write $x = \pstrzinf{u}$.
As before, there exists an $N > 0$ such that $\deg u_n =l $ for all $n \ge N$.
Let $n \gg N \gg l$.
It is sufficient to show that there exists a $y \in U(u_n)$ on the orbit of $x$
 such that $\fai_{n,N}(f^i(x)) \ne \fai_{n,N}(f^i(y))$ for infinitely many $i > 0$.
Because $c_{n+1,l}$ winds around $c_{n,l}$ twice,
 there exists a $v_{n+1}$ with $u_{n+1} \ne v_{n+1} \in V(c_{n+1,l})$
 such that $\fai_{n}(v_{n+1}) = u_n$.
Then, we can construct $y = \pstrzinf{v}$ with $v_i = u_i$ $(0 \le i \le n)$
 such that $\deg y = l$ and $\gap(u_i, v_i)$ is equal to some constant $i_n \ne 0$ for all $i \ge n$.
Therefore, we have constructed a $y_n \in U(u_n)$ with $f^{i_n}(x) = y_n$.
We must consider two cases:\\
Case 1: Suppose that both $\fai_{\infty,N}(f^i(x))$ and $\fai_{\infty,N}(f^i(y))$ with $i > 0$ trace $c_{N,l}$ only finitely many times.
Then, after tracing all $c_{N,l}$'s, they only trace $c_{N,l+k}$ with $k > 0$ and
 trace $c_{N,l+1}$ infinitely many times.
By Lemma \ref{lem:tail-gap-occurrence}, these occurrences are not periodic.
Therefore, we have the desired conclusion.\\
Case 2: Suppose that both $\fai_{\infty,N}(f^i(x))$
 and $\fai_{\infty,N}(f^i(y))$ with $i > 0$ trace $c_{N,l}$ infinitely many times.
By Lemma \ref{lem:usual-gap-occurrence}, these occurrences are not periodic.
Therefore, we also obtain the conclusion.\\
This completes the proof.
\end{proof}
We have finished the proof of Theorem \ref{thm:main}.
We suggest that, in defining the sequence of graph covers,
 the  expression
 $\fai_n(c_{n+1,i}) = e_{n,0} + n_{n,i} c_{n,i}+ n_{n,i+1} c_{n,i+1}
                      + \dotsb + n_{n,n} c_{n,n} + e_{n,0}$,
 with $n_{n,i} \ge 2$ for all $n \bni$ and $1 \le i \le n$ can be used.
The above proofs may also be applicable in this case.

\end{document}